\documentclass[10pt]{article}
\usepackage{amsmath, amsthm, amssymb, amsfonts, mathrsfs, mathtools}
\usepackage{graphicx}
\usepackage{epstopdf}
\usepackage{makeidx}
\usepackage[left=2cm,right=2cm,top=2cm,bottom=2cm]{geometry}
\usepackage{caption}
\usepackage{subcaption}
\usepackage[section]{placeins}
\usepackage{enumitem}
\usepackage{tikz}
\usetikzlibrary{decorations.pathreplacing}
\tikzstyle{vertex}=[auto=left,circle,draw=black,fill=white, inner sep=1.5]
\usepackage{hyperref}
\hypersetup{colorlinks=true, citecolor={blue}, linkcolor=blue, filecolor=magenta, urlcolor=cyan}
\newcommand{\Mod}[1]{\ (\mathrm{mod}\ #1)}
\newtheorem{theorem}{Theorem}[section]
\newtheorem{prop}[theorem]{Proposition}

\newtheorem{lema}[theorem]{Lemma}
\newtheorem{corollary}{Corollary}[theorem]

\title{Integral mixed circulant graph }

\author{Monu Kadyan\\
	Department of Mathematics\\
	Indian Institute of Technology Guwahati\\
	Guwahati, India - 781039\\
	Email: monu.kadyan@iitg.ac.in\\
	\\
	Bikash Bhattacharjya\\
	Department of Mathematics\\
	Indian Institute of Technology Guwahati\\
	Guwahati, India - 781039\\
	Email: b.bikash@iitg.ac.in 
}

\date{}

\begin{document}
	\maketitle
	
	\vspace{-0.3in}
	
	
	\begin{center}{\textbf{Abstract}}\end{center}

	A mixed graph is said to be \textit{integral} if all the eigenvalues of its Hermitian adjacency matrix are integer. The \textit{mixed circulant graph} $Circ(\mathbb{Z}_n,\mathcal{C})$ is a mixed graph on the vertex set $\mathbb{Z}_n$ and edge set $\{ (a,b): b-a\in \mathcal{C} \}$, where $0\not\in \mathcal{C}$.  If $\mathcal{C}$ is closed under inverse, then $Circ(\mathbb{Z}_n,\mathcal{C})$ is called a \textit{circulant graph}. We express the eigenvalues of $Circ(\mathbb{Z}_n,\mathcal{C})$ in terms of primitive $n$-th roots of unity, and find a sufficient condition for integrality of the eigenvalues of $Circ(\mathbb{Z}_n,\mathcal{C})$. For $n\equiv 0 \Mod 4$, we factorize the cyclotomic polynomial into two irreducible factors over $\mathbb{Q}(i)$. Using this factorization, we characterize integral mixed circulant graphs in terms of its symbol set. We also express the integer eigenvalues of an integral oriented circulant graph in terms of a Ramanujan type sum, and discuss some of their properties. 
		
\vspace*{0.3cm}

\section{Introduction}

	We only consider graphs without loops and  multi-edges. A  \textit{(simple) graph} $G$ is denoted by $G=(V(G),E(G))$, where $V(G)$ is the vertex set of $G$, and the edge set  $E(G)\subset V(G) \times V(G)\setminus \{(u,u)|u\in V(G)\}$ with $(u,v)\in E(G)$ if and only if $(v,u)\in E(G)$. 
	A graph $G$ is said to be \textit{oriented} if $(u,v)\in E(G)$ implies that $(v,u)\not\in E(G)$. A graph $G$ is said to be \textit{mixed} if $(u,v)\in E(G)$ does not always imply that $(v,u)\in E(G)$, see \cite{2015mixed} for details. In a mixed graph $G$, we call an edge with end vertices $u$ and $v$ to be \textit{undirected} (resp. \textit{directed}) if both $(u,v)$ and $(v,u)$ belong to $E(G)$ (resp. only one of $(u,v)$ and $(v,u)$ belongs to $E(G)$). An undirected edge $(u,v)$ is denoted by $u\leftrightarrow v$, and a directed edge $(u,v)$ is denoted by $u\rightarrow v$. Hence, a mixed graph can have both directed and undirected edges. Note that, if all edges of a mixed graph G are directed (resp. undirected) then G is an oriented graph (resp. a simple graph). For a mixed graph $G$, the underlying graph $G_{U}$ of $G$ is the simple undirected graph in which all edges of $G$ are considered undirected. By the terms of order, size, number of components, degree of a vertex, distance between two vertices etc., we mean that they are the same as in their underlying graphs.
	
	In 2015, Liu and Li \cite{2015mixed} introduced Hermitian adjacency matrix of a mixed graph. For a mixed graph with $n$ vertices, its \textit{Hermitian adjacency matrix} is denoted by $H(G)=(h_{uv})_{n\times n}$, where $h_{uv}$ is given by
	$$h_{uv} = \left\{ \begin{array}{rl}
		1 &\mbox{ if }
		(u,v)\in E \textnormal{ and } (v,u)\in E, \\ i & \mbox{ if } (u,v)\in E \textnormal{ and } (v,u)\not\in E ,\\
		-i & \mbox{ if } (u,v)\not\in E \textnormal{ and } (v,u)\in E,\\
		0 &\textnormal{ otherwise.}
	\end{array}\right.$$ 
	Here $i=\sqrt{-1}$ is the imaginary number unit. Hermitian adjacency matrix of a mixed graph incorporates both adjacency matrix of simple graph and skew adjacency matrix of an oriented graph. The Hermitian spectrum of $G$, denoted by $Sp_H(G)$, is the multi set of the eigenvalues of $H(G)$. It is easy to see that $H(G)$ is a Hermitian matrix and so $Sp_H(G)\subseteq \mathbb{R}$. 
	
	A mixed graph is said to be \textit{integral} if all the eigenvalues of its Hermitian adjacency matrix are integers. Integral graphs were first defined by Harary and Schwenk in 1974 \cite{harary1974graphs} and proposed a classification of integral graphs. See \cite{balinska2002survey} for a survey on integral graphs. Watanabe and Schwenk \cite{watanabe1979note,watanabe1979integral} proved several interesting results on integral trees in 1979. Csikvari \cite{csikvari2010integral} constructed integral trees with arbitrary large diameters in 2010. Further research on integral trees can be found in \cite{brouwer2008integral,brouwer2008small,wang2000some, wang2002integral}. In $2009$, Ahmadi et al. \cite{ahmadi2009graphs} proved that only a fraction of $2^{-\Omega (n)}$ of the graphs on $n$ vertices have an integral spectrum. Bussemaker et al. \cite{bussemaker1976there} proved that there are exactly $13$ connected cubic integral graphs. Stevanovi{\'c} \cite{stevanovic20034} studied the $4$-regular integral graphs avoiding $\pm3$ in the spectrum, and Lepovi{\'c} et al. \cite{lepovic2005there} proved that there are $93$ non-regular, bipartite integral graphs with maximum degree four. 
	In 2017, Guo et. al. \cite{2017mixed}  found all possible mixed graphs on $n$ vertices with spectrum $(-n+1,-1,-1,...,-1)$.
	
	Let $\Gamma$ be a group, $\mathcal{C} \subseteq \Gamma$ and $\mathcal{C}$ does not contain identity element of $\Gamma$. The set $\mathcal{C}$ is said to be \textit{symmetric} (resp. \textit{skew-symmetric}) if $\mathcal{C}$ is closed under inverse (resp. $a^{-1} \not\in\mathcal{C}$ for all $a\in \mathcal{C}$). Define $\overline{\mathcal{C}}= \{u\in \mathcal{C}: u^{-1}\not\in \mathcal{C} \}$. Clearly $\mathcal{C}\setminus \overline{\mathcal{C}}$ is symmetric and $\overline{\mathcal{C}}$ is skew-symmetric. The \textit{mixed Cayley graph} $G=Cay(\Gamma,\mathcal{C})$ is a mixed graph, where $V(G)=\Gamma$ and $E(G)=\{ (a,b):a,b\in \Gamma, ba^{-1}\in \mathcal{C} \}$.  Since we have not assumed that $\mathcal{C}$ is symmetric, so mixed Cayley graph can have directed edges. If $\mathcal{C}$ is symmetric, then $G$ is a (simple) \textit{Cayley graph}. If $\mathcal{C}$ is skew-symmetric then $G$ is an \textit{oriented Cayley graph}.

	In 1982, Bridge and Mena \cite{bridges1982rational} introduced a characterization of integral Cayley graphs over abelian group. Later on, same characterization was rediscovered by Wasin So \cite{2006integral} for cyclic groups in 2005.
	In 2009, Abdollahi and Vatandoost \cite{abdollahi2009cayley} proved that there are exactly seven connected cubic integral Cayley graphs. On the same year, Klotz and Sander \cite{klotz2010integral} proved that if Cayley graph $Cay(\Gamma,\mathcal{C})$ over abelian group $\Gamma$ is integral then $\mathcal{C}$ belongs to the Boolean algebra $\mathbb{B}(\Gamma)$ generated by the subgroups of $\Gamma$. Moreover, they conjectured that the converse is also true, which has been proved by Alperin and Peterson \cite{alperin2012integral}. In 2014, Cheng et al. \cite{ku2015cayley} proved that normal Cayley graphs (its generating set $\mathcal{C}$ is closed under conjugation) over symmetric groups are integral. Alperin \cite{alperin2014rational} gave a characterization of integral Cayley graphs over finite groups.  In 2017, Lu et al. \cite{lu2018integral} gave necessary and sufficient condition for the integrality of Cayley graphs over dihedral group $D_n=\langle a,b| a^n=b^2=1, bab=a^{-1} \rangle$. In particular, they completely determined all integral Cayley graphs over the dihedral group $D_p$ for a prime $p$. In 2019, Cheng et al.\cite{cheng2019integral} obtained several simple sufficient conditions for the integrality of Cayley graphs over dicyclic group $T_{4n}= \langle a,b| a^{2n}=1, a^n=b^2,b^{-1}ab=a^{-1} \rangle $. In particular, they also completely determined all integral Cayley graphs over the dicyclic group $T_{4p}$ for a prime $p$. However, it is far from being solved to obtain an explicit characterization of integral mixed Cayley graphs. As a simple attempt to this aspect, in this present paper, we give a characterization of integral mixed Cayley graph over cyclic group. If $\Gamma=\mathbb{Z}_n$ then the mixed Cayley graph $Cay(\mathbb{Z}_n,\mathcal{C})$ is known as mixed circulant graph, we denote it by $Circ(\mathbb{Z}_n,\mathcal{C})$. If $\mathcal{C}$ is symmetric then $Circ(\mathbb{Z}_n,\mathcal{C})$ is a (simple) \textit{circulant graph}. If $\mathcal{C}$ is skew-symmetric then $Circ(\mathbb{Z}_n,\mathcal{C})$ is an \textit{oriented circulant graph}.
	
This paper is organized as follows. In second section, we express the eigenvalues of a mixed circulant graph as a sum of eigenvalues of a simple circulant graph and an oriented Cayley graph. In third section, for $n\equiv 0 \Mod 4$, we obtain a sufficient condition on the set $\mathcal{C}$ for integral mixed circulant graphs. In fourth section, for $n\equiv 0 \Mod 4$, we obtain two monic irreducible factors $\Phi_n^1(x)$ and $\Phi_n^3(x)$ of the cyclotomic polynomial $\Phi_n(x)$ over $\mathbb{Q}(i)$ such that $\Phi_n(x)=\Phi_n^1(x).\Phi_n^3(x)$. Note that for $n\not\equiv 0 \Mod 4$, the cyclotomic polynomial $\Phi_n(x)$ is irreducible over $\mathbb{Q}(i)$. We prove, for even positive integer $n$, that $ n/2= \sum\limits_{d\in D_n} \varphi ( n/d)$, where $D_n$ is the set of all odd divisors of $n$. In fifth section, we characterize integral mixed circulant graphs in terms of its symbol set. In the last section, we express the integer eigenvalues of an integral oriented circulant graph in terms of a Ramanujan type sum, and discuss some of their properties. 

	\section{Oriented and mixed circulant graph}

	An $n\times n$ \textit{circulant matrix} $C$ have the form 
	
	\begin{equation*}
		C = 
		\begin{pmatrix}
			c_{0} & c_{1} &c_2& \cdots & c_{n-1} \\
			c_{n-1} & c_{0} &c_1& \cdots & c_{n-2} \\
			c_{n-2} & c_{n-1} &c_0& \cdots & c_{n-3} \\
			\vdots & \vdots & \vdots & \ddots & \vdots \\
			c_{1} & c_{2} &c_3& \cdots & c_{0} 
		\end{pmatrix},
	\end{equation*} where each row is a cyclic shift of the row above it. We see that the $(j,k)$-th entry of $C$ is $c_{k-j\Mod n}$. We denote such a circulant matrix by $C=Circ(c_0,...,c_{n-1})$. 
	
	The circulant matrix $C$ is diagonalizable by the matrix $F$ whose $k$-th column is defined by $$F_k= \frac{1}{\sqrt{n}}[1,w_n^{k},...,w_n^{(n-1)k} ]^t,$$ where $w_n=\exp(2\pi i /n)$ and $0\leq k \leq (n-1)$. The eigenvalues of $C$ are 
	$$\lambda_k=\sum\limits_{l=0}^{n-1}c_l w_n^{lk}, \textnormal{ for } k=0,1,...,n-1.$$ 
	
	For more information on circulant matrices see \cite{circulant}. 
	Let $G=Circ(\mathbb{Z}_n,\mathcal{C})$ be a mixed circulant graph on $n$ vertices. Let $c_{j\Mod n}=h_{1, j+1}$. We see that $h_{jk}=c_{k-j \Mod n}$. So the Hermitian adjacency matrix of a mixed circulant graph is a circulant matrix.

	\begin{lema}\label{iocg03} The spectrum of the mixed circulant graph $Circ(\mathbb{Z}_n,\mathcal{C})$ is $\{\gamma_0,\gamma_1,...,\gamma_{n-1} \}$, where $\gamma_j=\lambda_j+\mu_j$, $$\lambda_j= \sum\limits_{k\in \mathcal{C}\setminus \overline{\mathcal{C}}} w_n^{jk} \textnormal{ and } \mu_j= i \sum\limits_{k\in \overline{\mathcal{C}}} (w_n^{jk} -w_n^{-jk}), \textnormal{ for } j=0,1,...,n-1.$$
	\end{lema}
	\begin{proof}
		Since Hermitian adjacency matrix of a mixed circulant graph is circulant, so the eigenvalues of $H(G)$ are  $$\gamma_j=\sum\limits_{k\in \mathcal{C}\setminus \overline{\mathcal{C}}} w_n^{jk}+ i \sum\limits_{k\in \overline{\mathcal{C}}} w_n^{jk} -i \sum\limits_{k\in \overline{\mathcal{C}}} w_n^{-jk}=\lambda_j+\mu_j.$$
	\end{proof}
	
	Observe that $\lambda_j=\lambda_{n-j}$ and $\mu_j=-\mu_{n-j}$ for all $j=0,1,...,n-1$. Next two corollaries are special case of Lemma~\ref{iocg03}.
	
	\begin{corollary}\label{iocg01}\cite{biggs1993algebraic} The spectrum of the circulant graph $Circ(\mathbb{Z}_n,\mathcal{C})$ is $\{\lambda_0,\lambda_1,...,\lambda_{n-1} \}$, where $$\lambda_j=\sum\limits_{k\in \mathcal{C}} w_n^{jk}, \textnormal{ for } j=0,1,...,n-1.$$
	\end{corollary}

	\begin{corollary}\label{iocg02} The spectrum of the oriented circulant graph $Circ(\mathbb{Z}_n,\mathcal{C})$ is $\{\mu_0,\mu_1,...,\mu_{n-1} \}$, where $$\mu_j= i \sum\limits_{k\in {\mathcal{C}}} (w_n^{jk} - w_n^{-jk}), \textnormal{ for } 0\leq j\leq n-1.$$
	\end{corollary}
	
	From Lemma~\ref{iocg03}, we observe that the eigenvalues of a mixed circulant graph $Circ(\mathbb{Z}_n,\mathcal{C})$ are the sum of the eigenvalues of the circulant graph $Circ(\mathbb{Z}_n,\mathcal{C} \setminus \overline{\mathcal{C}})$ and the oriented circulant graph $Circ(\mathbb{Z}_n,\overline{\mathcal{C}})$. 
	
	Let $n\geq 2$ be a fixed positive integer. We review some basic definitions and notations from Wasin So \cite{2006integral}. For a divisor $d$ of $n$, define $M_n(d)=\{ d,2d,...,n-d\}$, the set of all positive multiples of $d$ less than $n$. Also define $G_n(d)=\{k: 1\leq k\leq n-1, \gcd(k,n)=d \}$. It is clear that $M_n(n)=G_n(n)=\emptyset$, $M_n(d)=dM_{(n/d)}(1)$ and $G_n(d)=dG_{(n/d)}(1)$.
	
	\begin{lema} \cite{2006integral}\label{icg1} If $n=dg$ for some $d,g\in \mathbb{Z}$ then $M_n(d)=\bigcup\limits_{h\mid g} G_n(hd)$.
	\end{lema}
	
	\begin{theorem} \cite{2006integral}\label{icg2} Let $d$ be a proper divisor of $n$. If $x^n=1$ then $\sum\limits\limits_{q\in G_n(d)}x^q\in \mathbb{Z}$.
	\end{theorem}
	
	Wasin So \cite{2006integral} characterized integral circulant graphs in the following theorem. Several results of the paper are influenced by \cite{2006integral}.

	\begin{theorem}\cite{2006integral}\label{2006integral}
		The circulant graph $Circ(\mathbb{Z}_n,\mathcal{C})$ is integral if and only if $\mathcal{C} =\bigcup\limits_{d\in \mathscr{D}}G_n(d)$, where $\mathscr{D} \subseteq \{ d :d\mid n\}$.
	\end{theorem}
	
	The next lemma will help us to extend characterization of integral oriented circulant graphs to integral mixed circulant graphs.
	
	\begin{lema}\label{iocg24}
		The mixed circulant graph $Circ(\mathbb{Z}_n,\mathcal{C})$ is integral if and only if both $Circ(\mathbb{Z}_n,\mathcal{C}\setminus \overline{\mathcal{C}})$ and $Circ(\mathbb{Z}_n, \overline{\mathcal{C}})$ are integral.
	\end{lema}
	\begin{proof}
		Let $\gamma_j$ be an integer eigenvalue of the mixed circulant graph $Circ(\mathbb{Z}_n,\mathcal{C})$. By Lemma~\ref{iocg03}, $\gamma_j=\lambda_j+\mu_j$ and $\gamma_{n-j}=\lambda_j-\mu_j$ for all $j=0,...,n-1$, where $\lambda_j$ is an eigenvalue of $Circ(\mathbb{Z}_n,\mathcal{C}\setminus \overline{\mathcal{C}})$ and $\mu_j$ is an eigenvalue of $Circ(\mathbb{Z}_n, \overline{\mathcal{C}})$. Thus $\lambda_j=\frac{\gamma_j+\gamma_{n-j}}{2} \in \mathbb{Q}$ and $\mu_j=\frac{\gamma_j-\gamma_{n-j}}{2} \in \mathbb{Q}$. As $\lambda_j$ and $\mu_j$ are algebraic integers, $\lambda_j, \mu_j \in \mathbb{Q}$ implies that $\lambda_j$ and $\mu_j$ are integers. Thus the circulant graph $Circ(\mathbb{Z}_n,\mathcal{C}\setminus \overline{\mathcal{C}})$ and the oriented circulant graph $Circ(\mathbb{Z}_n, \overline{\mathcal{C}})$ are integral. 
		
		Conversely, assume that both $Circ(\mathbb{Z}_n,\mathcal{C}\setminus \overline{\mathcal{C}})$ and $(\mathbb{Z}_n, \overline{\mathcal{C}})$ are integral. Then Lemma \ref{iocg03} implies that $Circ(\mathbb{Z}_n,\mathcal{C})$ is integral. 
	\end{proof}
	
	Since some characterization of integral circulant graphs are known from \cite{2006integral}, in view of Lemma ~\ref{iocg24}, it is enough to find characterization of oriented circulant graph. Note that $x^n=1$ if and only if $x\in \{w_n^t: 0\leq t\leq n-1 \}$, where $w_n=\exp(\frac{2\pi i}{n})$. Therefore, the oriented circulant graph $Circ(\mathbb{Z}_n,\mathcal{C})$ is integral if and only if $\sum\limits_{k\in \mathcal{C}} i(x^k-x^{(n-k)}) \in \mathbb{Z}$ for all $x$ such that $x^n=1$.

	\section{Sufficient condition for integral mixed circulant graph on $n\equiv 0 \Mod 4$ vertices}
	
	Let $n\equiv 0 \Mod 4$ be a fixed positive integer throughout this section. For a divisor $d$ of $\frac{n}{4}$ and $r\in \{0,1,2,3\}$, define $M_n^r(d)=\{dk: 0\leq dk < n , k \equiv r \Mod 4 \}$. For $g\in \mathbb{Z}$, let $D_g$ be the set of all odd divisors of $g$.
	
	\begin{lema}\label{iocg1} Let $n=4dg$ for some $d,g\in \mathbb{Z}$. Then the following are true:
		\begin{enumerate}
			\item[(i)] $M_n^1(d) \bigcup M_n^3(d)=\bigcup\limits_{h\in D_g} G_n(hd) $.
			\item[(ii)] $M_n^2(d)=\bigcup\limits_{h\in D_g} G_n(2hd) $.
			\item[(iii)] $M_n^0(d)= M_n(4d)\bigcup \{ 0\}$.
		\end{enumerate}
	\end{lema}
	\begin{proof} 
		\begin{enumerate}
			\item[(i)] Let $dk\in M_n^1(d) \bigcup\limits M_n^3(d)$. Since $k\equiv 1$ or $3 \Mod 4$, it follows that $k$ is an odd integer. So $\gcd (k,g)=h\in D_g$. Hence $\gcd (kd,n)=\gcd (kd, 4dg)=hd$ implies that $M_n^1(d) \bigcup M_n^3(d)\subseteq \bigcup\limits_{h\in D_g} G_n(hd)$. On the other hand, if $ x\in \bigcup\limits_{h\in D_g} G_n(hd)$ then there exists $h\in D_g$ such that $\gcd(x,n)=hd$. So there exists $x_0\in \mathbb{Z}$ such that $x=hdx_0$ and $\gcd (x_0,\frac{n}{hd})=1$. Since $\frac{n}{hd}$ is an even integer and $\gcd (x_0,\frac{n}{hd})=1$, so $x$ is an odd multiple of $d$. Note that $ M_n^1(d) \bigcup M_n^3(d)$ is the set of all odd multiples of $d$ between $0$ and $n$. Thus $x\in M_n^1(d) \bigcup M_n^3(d)$. Hence $\bigcup\limits_{h\in D_g} G_n(hd) \subseteq M_n^1(d) \bigcup M_n^3(d)$.
			
			\item[(ii)] Let $dk\in M_n^2(d)$ and $h=\gcd(\frac{k}{2},g)$. Since $k\equiv 2 \Mod 4$, we see that $\gcd (k,4g)=2h$. This gives $\gcd (kd,n)=\gcd (kd, 4dg)=2hd$, and so $M_n^2(d) \subseteq \bigcup\limits_{h\in D_g} G_n(2hd)$. On the other hand, if $ x\in \bigcup\limits_{h\in D_g} G_n(2hd)$, there exists $h\in D_g$ such that $\gcd(x,n)=2hd$. So there exists $x_0\in \mathbb{Z}$ such that $x=2hdx_0$ and $\gcd (x_0,\frac{n}{2hd})=1$. Since $\frac{n}{2hd}$ is an even integer and $\gcd (x_0,\frac{n}{2hd})=1$, so $x_0$ is an odd integer. Now $h,x_0$ are odd integers, and so letting $\alpha=2hx_0$, we get $x=\alpha d$, where $\alpha \equiv 2 \Mod 4$. Hence $x\in M_n^2(d)$. Thus $\bigcup\limits_{h\in D_g} G_n(2hd) \subseteq M_n^2(d)$.
			\item[(iii)] By definition, we have
			\begin{equation*}
				\begin{split}
					M_n^0(d)&=\{ dk: 0\leq dk \leq n-1, k=4\alpha \textnormal{ for some } \alpha \in \mathbb{Z}\}\\
					&= \{ 4\alpha d : 1\leq 4 \alpha d \leq n-1 \textnormal{ for some } \alpha \in \mathbb{Z}\}\cup \{ 0\}\\
					&= M_n(4d) \cup \{0\}.
				\end{split} 
			\end{equation*}
		\end{enumerate}
	\end{proof}
	
	\begin{lema}\label{iocg2} If $x^n=1$ and $n=4dg$ for some $d,g\in \mathbb{Z}$, then $ \sum\limits_{q\in M_n^r(d)} x^q\in \mathbb{Z} $, for $r=0,2$. 
	\end{lema}
	\begin{proof} From Lemma~\ref{iocg1} and Lemma~\ref{icg1}, we have $M_n^0(d)=\bigcup\limits_{h\mid g} G_n(4hd)\bigcup\limits \{0\}$. Therefore, 
		\begin{equation}\label{eq1}
			\begin{split}
				\sum\limits_{q\in M_n^0(d)} x^q =1+\sum\limits_{h\mid g} \sum\limits_{q\in G_n(4hd)} x^q.
			\end{split} 
		\end{equation}
		As $G_n(4hd)=G_n(n)=\emptyset$, the summation in (\ref{eq1}) is over proper divisors of $g$. Theorem~\ref{icg2} now implies that $\sum\limits_{q\in M_n^0(d)} x^q\in \mathbb{Z} $. Similarly, one can show that $ \sum\limits_{q\in M_n^2(d)} x^q\in \mathbb{Z}$.
	\end{proof}

	\begin{lema}\label{iocg3}
		If $x^n=1$ and $n=4dg$ for some $d,g\in \mathbb{Z}$, then $$i\bigg( \sum\limits_{q\in M_n^1(d)} x^q\bigg)- i\bigg( \sum\limits_{q\in M_n^3(d)} x^q\bigg) \in \mathbb{Z}.$$
	\end{lema}
	\begin{proof}We have
		\begin{equation*}
			\begin{split}
				(x^n-1)&=(x^d-i)\bigg(\sum\limits_{k=1}^{4g} i^{k-1} x^{n-dk} \bigg)\\
				&=(x^d-i)\bigg(\sum\limits_{k=1}^{4g} i^{k} x^{n-dk} \bigg) (-i)\\
				&=(x^d-i) \bigg( \sum\limits_{\substack{1\leq k \leq 4g \\ k\equiv 0 \Mod 4}} x^{n-dk} - \sum\limits_{\substack{1\leq k \leq 4g \\ k\equiv 3 \Mod 4}} i x^{n-dk} - \sum\limits_{\substack{1\leq k \leq 4g \\ k\equiv 2 \Mod 4}} x^{n-dk}+ \sum\limits_{\substack{1\leq k \leq 4g \\ k\equiv 1 \Mod 4}} i x^{n-dk}\bigg) (-i)\\
				&=(x^d-i) \bigg( \sum\limits_{q\in M_n^0(d)} x^{q} - \sum\limits_{q\in M_n^1(d)} i x^{q} - \sum\limits_{q\in M_n^2(d)} x^{q}+ \sum\limits_{q\in M_n^3(d)} i x^{q}\bigg) (-i)\\  
				&=(x^d-i) \bigg( \sum\limits_{q\in M_n^1(d)} i x^{q} - \sum\limits_{q\in M_n^3(d)} i x^{q} - \sum\limits_{q\in M_n^0(d)} x^{q}+ \sum\limits_{q\in M_n^2(d)} x^{q}\bigg) (i).
			\end{split} 
		\end{equation*}
		Case 1. Assume that $x^d-i=0$.\\ If $q\in M_n^1(d)$, then $q=(4y_1+1)d$, for some $y_1\in \mathbb{Z}$. So $x^q= x^{4y_1d+d}=x^d.(x^d)^{4y_1}=i.i^{4y_1}=i$.\\
		If $q\in M_n^3(d)$, then $q=(4y_3+3)d$, for some $y_3\in \mathbb{Z}$. So $x^q= x^{4y_3d+3d}=(x^d)^3.(x^d)^{4y_3}=i^3.i^{4y_3}=-i$.\\ Thus
		$$i\bigg( \sum\limits_{q\in M_n^1(d)} x^q\bigg)- i\bigg( \sum\limits_{q\in M_n^3(d)} x^q\bigg)= i \bigg( \sum\limits_{q\in M_n^1(d)}i\bigg)- i\bigg( \sum\limits_{q\in M_n^3(d)}-i\bigg) =- |M_n^1(d) \cup M_n^3(d)| \in \mathbb{Z}.$$
		Case 2. Assume that $x^d-i\neq0$. Then $$ \sum\limits_{q\in M_n^1(d)} i x^{q} - \sum\limits_{q\in M_n^3(d)} i x^{q} - \sum\limits_{q\in M_n^0(d)} x^{q}+\sum\limits_{q\in M_n^2(d)} x^{q}=0, \textnormal{ as }x^n-1=0.$$ 
		
		$$\Rightarrow \sum\limits_{q\in M_n^1(d)} i x^{q} - \sum\limits_{q\in M_n^3(d)} i x^{q} = \sum\limits_{q\in M_n^0(d)} x^{q}- \sum\limits_{q\in M_n^2(d)} x^{q}\in \mathbb{Z}, \textnormal{ by Lemma } \ref{iocg2}.$$ 
	\end{proof}
	
	Let $n=4dg$ for some $d,g\in \mathbb{Z}$. Since $1\in D_g$, Lemma \ref{iocg1} gives $G_n(d)\subseteq M_n^1(d) \cup M_n^3(d)$. For $r\in \{1,3 \}$, define $$G_n^r(d)=\{ dk: k\equiv r \Mod 4, \gcd(dk,n )= d \}.$$ In the next lemma, we prove that $G_n(d)$ is a disjoint union of $G_n^1(d)$ and $G_n^3(d)$. For $r\in \{1,3 \}$, define $D_g^r$ to be the set of all odd divisors of $g$ which are congruent to $r$ modulo $4$. Note that $D_g=D_g^1 \bigcup\limits D_g^3$.

	\begin{lema}\label{iocg4} If $n=4dg$ for some $d,g\in \mathbb{Z}$, then the following hold:
		\begin{enumerate}[label=(\roman*)]
			\item $G_n^1(d) \cap G_n^3(d)=\emptyset$.
			\item $G_n(d)=G_n^1(d) \cup G_n^3(d)$.
			\item $M_n^1(d) =\bigg( \bigcup\limits_{h\in D_g^1} G_n^1(hd) \bigg) \cup \bigg( \bigcup\limits_{h\in D_g^3} G_n^3(hd) \bigg)$.
			\item $M_n^3(d) =\bigg( \bigcup\limits_{h\in D_g^1} G_n^3(hd) \bigg) \cup \bigg( \bigcup\limits_{h\in D_g^3} G_n^1(hd) \bigg)$.
		\end{enumerate}
	\end{lema}
	\begin{proof}
		\begin{enumerate}[label=(\roman*)]
			\item It is clear from the definition of $G_n^1(d)$ and $G_n^3(d)$ that $G_n^1(d) \cap G_n^3(d)=\emptyset$.
			\item Since $G_n^1(d) \subseteq G_n(d)$ and $G_n^3(d) \subseteq G_n(d)$, so $G_n^1(d) \cup G_n^3(d) \subseteq G_n(d)$. On the other hand, if $x\in G_n(d)$ then $x=d\alpha$ for some $\alpha$. Note that $\gcd(\alpha ,n )=1$, so $\alpha$ is an odd integer. Thus $\alpha \equiv 1\Mod 4$ or $\alpha \equiv 3\Mod 4$. Hence $G_n(d) \subseteq G_n^1(d) \cup G_n^3(d)$.
			\item Let $dk\in M_n^1(d)$ so that $k\equiv 1\Mod 4$. Lemma ~\ref{iocg1} gives an $h\in D_g$ satisfying $dk\in G_n(hd)$. Thus $dk \in G_n^1(hd)$ or $dk \in G_n^3(hd)$, by part $(ii)$.\\ 
			Case 1: Assume that $h\equiv 1\Mod 4$.\\Let, if possible, $dk \in G_n^3(hd)$, $i.e.$, $dk=\alpha hd$, where $\alpha \equiv 3 \Mod 4$ and $\gcd(\alpha, n)=1$. Then we have $k=\alpha h \equiv 3 \Mod 4$, a contradiction. Hence $dk \in G_n^1(hd)$.\\
			Case 2: Assume that $h\equiv 3\Mod 4$.\\ As in Case 1, we get $k\equiv 3 \Mod 4$, a contradiction. Hence $dk \in G_n^3(hd)$.\\Thus $M_n^1(d) \subseteq \bigg( \bigcup\limits_{h\in D_g^1} G_n^1(hd) \bigg) \cup \bigg( \bigcup\limits_{h\in D_g^3} G_n^3(hd) \bigg)$. On the other hand, if $\alpha hd \in G_n^1(hd)$, where $h\in D_g^1$ and $\alpha \equiv 1 \Mod 4$, we get $\alpha h \equiv 1 \Mod 4$, $i.e.$, $\alpha h d \in M_n^1(d)$. Similarly, $\beta h d \in G_n^3(hd) \Rightarrow \beta h d \in M_n^1(d)$. Therefore, $\bigcup\limits_{h\in D_g^1} G_n^1(hd) \cup \bigcup\limits_{h\in D_g^3} G_n^3(hd) \subseteq M_n^1(d)$.
			\item The proof of this part is similar to the proof of part (iii).
		\end{enumerate}
	\end{proof}

	Note that $G_n(d)=dG_{(n/d)}(1)$. Therefore $G_n(d)=G_n^1(d) \cup G_n^3(d)$ gives that $G_n^1(d)=dG_{(n/d)}^1(1)$ and $G_n^3(d)=dG_{(n/d)}^3(1)$.

	\begin{theorem}\label{iocg5} 
		Let $n\equiv 0 \Mod 4$ and $d_1=\frac{n}{4}> d_2>...>d_s=1$ be all positive divisors of $\frac{n}{4}$. If $x^n=1$, then $$ \sum\limits_{q\in G_n^1(d_k)} ix^q- \sum\limits_{q\in G_n^3(d_k)} ix^q \in \mathbb{Z} \textnormal{ for } k=1,2,...,s.$$
	\end{theorem}
	\begin{proof} Apply induction on $k$. For $k=1$, take $d_1=\frac{n}{4}$, $g_1=1$ in Lemma ~\ref{iocg4} to get $M_n^1(d_1)=G_n^1(d_1)$ and $M_n^3(d_1)=G_n^3(d_1)$.\\ Hence by Lemma ~\ref{iocg3}, $ \sum\limits_{q\in G_n^1(d_1)} ix^q- \sum\limits_{q\in G_n^3(d_1)} ix^q \in \mathbb{Z} $. \\
		Assume that the statement holds for $k$, where $1\leq k <s$. To show that $ \sum\limits_{q\in G_n^1(d_{k+1})} ix^q- \sum\limits_{q\in G_n^3(d_{k+1})} ix^q $ is an integer. Now $n=4d_{k+1}g_{k+1}$ for some $d_{k+1},g_{k+1}\in \mathbb{Z}$. Lemma ~\ref{iocg4} implies that
		\begin{equation}\label{eq2}
			\begin{split}
				M_n^1(d_{k+1}) =\bigg( \bigcup\limits_{h\in D_{g_{k+1}}^1} G_n^1(hd_{k+1}) \bigg) \cup \bigg( \bigcup\limits_{h\in D_{g_{k+1}}^3} G_n^3(hd_{k+1}) \bigg)
			\end{split} 
		\end{equation}
		and 
		\begin{equation}\label{eq3}
			\begin{split}
				M_n^3(d_{k+1}) =\bigg( \bigcup\limits_{h\in D_{g_{k+1}}^1} G_n^3(hd_{k+1}) \bigg)  \cup \bigg( \bigcup\limits_{h\in D_{g_{k+1}}^3} G_n^1(hd_{k+1}) \bigg).
			\end{split} 
		\end{equation}
		Note that $hd_{k+1}$ is also a divisor of $\frac{n}{4}$. If $h>1$ then $hd_{k+1}=d_j$ for some $j<k+1$ and so by induction hypothesis $$ \sum\limits_{q\in G_n^1(hd_{k+1})} ix^q- \sum\limits_{q\in G_n^3(hd_{k+1})} ix^q \in \mathbb{Z}.$$
		
		Note that the unions in (\ref{eq2}) and (\ref{eq3}) are disjoint unions. We have
		\begin{equation*}
			\begin{split}
				\sum\limits_{q\in M_n^1(d_{k+1})} ix^q- \sum\limits_{q\in M_n^3(d_{k+1})} ix^q &= \bigg( \sum\limits_{q\in G_n^1(d_{k+1})} ix^q- \sum\limits_{q\in G_n^3(d_{k+1})} ix^q \bigg) \\
				&+ \sum\limits_{h\in D_{g_{k+1}}^1, h> 1} \bigg( \sum\limits_{q\in G_n^1(hd_{k+1})} ix^q- \sum\limits_{q\in G_n^3(hd_{k+1})} ix^q \bigg)\\
				&+ \sum\limits_{h\in D_{g_{k+1}}^3, h> 1} \bigg( \sum\limits_{q\in G_n^3(hd_{k+1})} ix^q- \sum\limits_{q\in G_n^1(hd_{k+1})} ix^q \bigg).
			\end{split} 
		\end{equation*}
		Hence
		\begin{align}\label{eq4}
			\sum\limits_{q\in G_n^1(d_{k+1})} ix^q- \sum\limits_{q\in G_n^3(d_{k+1})} ix^q&=\bigg( \sum\limits_{q\in M_n^1(d_{k+1})} ix^q- \sum\limits_{q\in M_n^3(d_{k+1})} ix^q \bigg) \nonumber\\
			&- \sum\limits_{h\in D_{g_{k+1}}^1, h> 1} \bigg( \sum\limits_{q\in G_n^1(hd_{k+1})} ix^q- \sum\limits_{q\in G_n^3(hd_{k+1})} ix^q \bigg) \nonumber\\
			&- \sum\limits_{h\in D_{g_{k+1}}^3, h> 1} \bigg( \sum\limits_{q\in G_n^3(hd_{k+1})} ix^q- \sum\limits_{q\in G_n^1(hd_{k+1})} ix^q \bigg).
		\end{align} Observe that the first summand in the right of (\ref{eq4}) is an integer by Lemma ~\ref{iocg3}, and the other two summands are also integers by induction hypothesis. Hence $\sum\limits_{q\in G_n^1(d_{k+1})} ix^q- \sum\limits_{q\in G_n^3(d_{k+1})} ix^q \in \mathbb{Z}$. Thus the proof follows by induction.
	\end{proof}

	Now, using Corollary~\ref{iocg02} and Theorem~\ref{iocg5}, we get sufficient condition on the symbol set $\mathcal{C}$ for which the oriented circulant graph $Circ(\mathbb{Z}_n,\mathcal{C})$ is integral. 
	
	\begin{corollary}\label{iocg6} 
		Let $Circ(\mathbb{Z}_n,\mathcal{C})$ be an oriented circulant graph on $n\equiv 0 \Mod 4$ vertices. Let $\mathcal{C}= \bigcup\limits_{d\in \mathscr{D}}S_n(d)$, where $\mathscr{D} \subseteq \{ d: d\mid \frac{n}{4}\}$ and $S_n(d) = G_n^{1}(d)$ or $S_n(d)=G_n^{3}(d)$. Then $Circ(\mathbb{Z}_n,\mathcal{C})$ is integral.
	\end{corollary}
	
	Now, using Theorem~\ref{2006integral} and Corollary~\ref{iocg6}, we get sufficient condition on the symbol set $\mathcal{C}$ for which the mixed circulant graph $Circ(\mathbb{Z}_n,\mathcal{C})$ is integral.

	\begin{corollary}\label{iocg7} 
		Let $Circ(\mathbb{Z}_n,\mathcal{C})$ be a mixed circulant graph on $n\equiv 0 \Mod 4$ vertices. Let $\mathcal{C} \setminus \overline{\mathcal{C}} = \bigcup\limits_{d\in \mathscr{D}_1}G_n(d)$ and $\overline{\mathcal{C}} = \bigcup\limits_{d\in \mathscr{D}_2}S_n(d)$, where $\mathscr{D}_1 \subseteq \{ d: d\mid n\}$, $\mathscr{D}_2 \subseteq \{ d: d\mid \frac{n}{4}\}$,  $\mathscr{D}_1 \cap \mathscr{D}_2 = \emptyset$, and $S_n(d) = G_n^{1}(d)$ or $S_n(d)=G_n^{3}(d)$. Then $Circ(\mathbb{Z}_n,\mathcal{C})$ is integral.
	\end{corollary}

	\section{Irreducible factors of cyclotomic polynomial over $\mathbb{Q}(i)$} 
	
	The \textit{cyclotomic polynomial} $\Phi_n(x)$ is the monic polynomial whose zeros are the primitive $n^{th}$ root of unity. That is $$\Phi_n(x)= \prod_{a\in G_n(1)}(x-w_n^a),$$ where $w_n=\exp(\frac{2\pi i}{n})$. Clearly the degree of $\Phi_n(x)$ is $\varphi(n)$. See \cite{numbertheory} for more details on cyclotomic polynomial.
	
	\begin{lema}~\cite{numbertheory}
		$x^n-1=\prod\limits_{d\mid n} \Phi_d(x)$ and $\Phi_n(x)\in \mathbb{Z}[x]$.
	\end{lema}
	
	\begin{theorem}~\cite{numbertheory}
		The cyclotomic polynomial $\Phi_n(x)$ is irreducible in $\mathbb{Z}[x]$.
	\end{theorem}

	The polynomial $\Phi_n(x)$ is irreducible over $\mathbb{Q}(i)$ if and only if $[\mathbb{Q}(i,w_n) : \mathbb{Q}(i)]= \varphi(n)$. Also $ \mathbb{Q}(w_n)$ does not contain the number $i=\sqrt{-1}$ if and only if $n\not\equiv 0 \Mod 4$. Thus, if $n\not\equiv 0 \Mod 4$ then $[\mathbb{Q}(i,w_n):\mathbb{Q}(w_n) ]=2=[\mathbb{Q}(i), \mathbb{Q}]$, and therefore $$[\mathbb{Q}(i,w_n) : \mathbb{Q}(i)]=\frac{[\mathbb{Q}(i,w_n) : \mathbb{Q}(w_n)] \times [\mathbb{Q}(w_n) : \mathbb{Q}]}{ [\mathbb{Q}(i) : \mathbb{Q}]}= [\mathbb{Q}(w_n) : \mathbb{Q}]= \varphi(n).$$ Hence for $n\not\equiv 0 \Mod 4$, the polynomial $\Phi_n(x)$ is irreducible over $\mathbb{Q}(i)$. 
	
	Let $n \equiv 0 \Mod 4$. Then $ \mathbb{Q}(i,w_n)= \mathbb{Q}(w_n)$, and so $$[\mathbb{Q}(i,w_n) : \mathbb{Q}(i)] = \frac{[\mathbb{Q}(i,w_n) : \mathbb{Q}]}{[\mathbb{Q}(i) : \mathbb{Q}]}=\frac{\varphi(n)}{2}.$$ Hence the polynomial $\Phi_n(x)$ is not irreducible over $\mathbb{Q}(i)$.
	
	Now we factorize $\Phi_n(x)$ into two irreducible monic polynomials over $\mathbb{Q}(i)$. From Lemma~\ref{iocg4}, we know that $G_n(1)$ is a disjoint union of $G_n^1(1)$ and $G_n^3(1)$. Define $$\Phi_n^{1}(x)= \prod_{a\in G_n^1(1)}(x-w_n^a) \textnormal{ and } \Phi_n^3(x)= \prod_{a\in G_n^3(1)}(x-w_n^a).$$ It is clear from the definition that $\Phi_n(x)=\Phi_n^1(x)\Phi_n^3(x)$.

	\begin{lema}\label{factor} If $n\equiv 0\Mod 4$, then the following are true:
		\begin{enumerate}
			\item[(i)] $$x^{\frac{n}{4}} -i= \prod_{d\in D_n^1}\Phi_{n/d}^1(x) \prod_{d\in D_n^3} \Phi_{n/d}^3(x);$$
			\item[(ii)] $$x^{\frac{n}{4}} +i= \prod_{d\in D_n^1}\Phi_{n/d}^3(x) \prod_{d\in D_n^3} \Phi_{n/d}^1(x).$$
		\end{enumerate}
	\end{lema}
	\begin{proof} 
		\begin{enumerate}
			\item[(i)] Note that $|M_n^1(1)|=\frac{n}{4}$ and $w_n^a$ is a root of $x^{\frac{n}{4}} -i $, for all $a\in M_n^1(1)$. Therefore
			\begin{equation*}
				\begin{split}
					x^{\frac{n}{4}} -i &= \prod_{a\in M_n^1(1)} (x-w_n^a)\\
					&= \prod_{h\in D_n^1} \prod_{a\in G_n^1(h)}(x-w_n^a) \prod_{h\in D_n^3} \prod_{a\in G_n^3(h)}(x-w_n^a), \hspace{1cm} (\textnormal{ Using lemma } ~\ref{iocg4})\\
					&= \prod_{h\in D_n^1} \prod_{a\in hG_{n/h}^1(1)}(x-w_n^a) \prod_{h\in D_n^3} \prod_{a\in hG_{n/h}^3(1)}(x-w_n^a)\\
					&= \prod_{h\in D_n^1} \prod_{a\in G_{n/h}^1(1)}(x-(w_n^h)^a) \prod_{h\in D_n^3} \prod_{a\in G_{n/h}^3(1)}(x-(w_n^h)^a)\\
					&= \prod_{h\in D_n^1}\Phi_{n/h}^1(x) \prod_{h\in D_n^3} \Phi_{n/h}^3(x).
				\end{split} 
			\end{equation*} 
			In last line, we are using the fact that $w_n^h=\exp(\frac{2\pi i}{n/h})$ is a primitive $\frac{n}{h}$-th root of unity.
			\item[(ii)] Note that $|M_n^3(1)|=\frac{n}{4}$ and $w_n^a$ is a root of $x^{\frac{n}{4}} +i $, for all $a\in M_n^3(1)$. Now proceed as in part $(i)$ to get the proof of this part.
		\end{enumerate}
	\end{proof}

	\begin{corollary}\label{monic}
		Let $n\equiv 0\Mod 4$. The factors $\Phi_n^1(x)$ and $\Phi_n^3(x)$ of $\Phi_n(x)$ are monic polynomials in $\mathbb{Z}(i)[x]$ of degree $\varphi(n)/2$.
	\end{corollary} 
	\begin{proof}
		By definition, $\Phi_n^1(x)$ and $\Phi_n^3(x)$ are monic. Also, $G_n(1)=G_n^1(1)\cup G_n^3(1)$ is a disjoint union and that $|G_n^1(1)|=|G_n^3(1)|$. Therefore the degree of $\Phi_n^1(x)$ and $\Phi_n^3(x)$ is $\varphi(n)/2$. Now we proceed by induction on $n$ to show that $\Phi_n^1(x), \Phi_n^3(x)\in \mathbb{Z}(i)[x]$. For $n=4$, the polynomials $\Phi_4^1(x)= x+i$ and $\Phi_4^3(x)=x-i$ are clearly in $\mathbb{Z}(i)[x]$. Assume  that $\Phi_k^1(x)$ and $\Phi_k^3(x)$ are in $\mathbb{Z}(i)[x]$ for $k<n$ and $k\equiv 0\Mod 4$. By Lemma~\ref{factor}, $\Phi_n^1(x) = \frac{x^{\frac{n}{4}} -i}{f(x)}$ and $\Phi_n^3(x) = \frac{x^{\frac{n}{4}} +i}{g(x)}$, where $f(x)$ and $g(x)$ are monic polynomials. By induction hypothesis $f(x),g(x) \in \mathbb{Z}(i)[x]$. It follows by ``long division'' that $\Phi_n^1(x)\in \mathbb{Z}(i)[x]$ and $\Phi_n^3(x)\in \mathbb{Z}(i)[x]$.
	\end{proof}
	
	\begin{corollary}\label{coronum}
		Let $n\equiv 0\Mod 4$. Then $$ \frac{n}{2}= \sum\limits_{d\in D_n} \varphi \bigg( \frac{n}{d} \bigg).$$
	\end{corollary} 
	\begin{proof}
		We get the desired equality by comparing the degree of the polynomials in Lemma~\ref{factor}.
	\end{proof}
	
	\begin{corollary}
		Let $n$ be an positive even integer. Then $$ \frac{n}{2}= \sum\limits_{d\in D_n} \varphi \bigg( \frac{n}{d} \bigg).$$
	\end{corollary} 
	\begin{proof}
		Since $2n\equiv 0\Mod 4$ and $\frac{2n}{d}\equiv 0 \Mod 4$ for all $d\in D_n$, so $\varphi ( \frac{2n}{d} )=2\varphi ( \frac{n}{d} )$ for $d\in D_n=D_{2n}$. Thus by Corollary~\ref{coronum}, $ \frac{2n}{2}= \sum\limits_{d\in D_{2n}} \varphi ( \frac{2n}{d} )=\sum\limits_{d\in D_{n}} 2 \varphi ( \frac{n}{d} ).$
	\end{proof}
	
	\begin{corollary}
		Let $n=2^k$ for $k\geq 2$. Then $ \Phi_n^1(x)=x^{\frac{n}{4}} +i \textnormal{ and } \Phi_n^3(x)=x^{\frac{n}{4}} -i. $
	\end{corollary}

	\begin{theorem}
		Let $n\equiv 0\Mod 4$. The factors $\Phi_n^1(x)$ and $\Phi_n^3(x)$ of $\Phi_n(x)$ are irreducible monic polynomials in $\mathbb{Q}(i)[x]$ of degree $\frac{\varphi(n)}{2}$.
	\end{theorem}
	\begin{proof}
		In view of Corollary~\ref{monic}, we only need to show that $\Phi_n^1(x)$ and $\Phi_n^3(x)$ are irreducible in $\mathbb{Q}(i)$. For $n\equiv 0\Mod 4$, we have $[\mathbb{Q}(i,w_n) : \mathbb{Q}(i)] = \frac{\varphi(n)}{2}$. So there is a unique irreducible monic polynomial $p(x)\in \mathbb{Q}(i)[x]$ of degree $\frac{\varphi(n)}{2}$ having $w_n$ as a root. Since $w_n$ is also a root of $\Phi_n^1(x)$, so $\Phi_n^1(x)=p(x). f(x)$, where $f(x)\in \mathbb{Q}(i)[x]$. As, $\Phi_n^1(x)$ is a monic polynomial of degree $\varphi(n)/2$, so $f(x)=1$. Hence $p(x)=\Phi_n^1(x)$. Similarly, for $a\in G_n^3(1)$, we have $[\mathbb{Q}(i,w_n^a) : \mathbb{Q}(i)] = \frac{\varphi(n)}{2}$. This along with Corollary~\ref{monic}, give that $\Phi_n^3(x)$ is irreducible.
	\end{proof}

	\section{Characterization of mixed circulant graphs}

	For a fixed $n\equiv 0\Mod 4$, $n\geq 4$, let $d\mid \frac{n}{4}$. Define $u(d)$ to be the $(n-1)$-vector, where
	
	$$u(d)_k= \left\{ \begin{array}{rl}
		1 & \mbox{if } k\in G_n^1(d) \\
		-1 & \mbox{if } k\in G_n^3(d)\\ 
		0 &   \mbox{otherwise}. 
	\end{array}\right.
	$$ Let $E=[e_{st}]$ be an $(n-1)\times(n-1)$ matrix defined by $$e_{st}=iw_n^{st}.$$ Note that $E$ is invertible. Using Theorem ~\ref{iocg5}, we get $E u(d) \in \mathbb{Z}^{n-1}$. 
	Define $\mathbb{U}=\{ u\in \mathbb{Q}^{n-1}: Eu\in \mathbb{Q}^{n-1} \}$, so that $u(d)\in \mathbb{U}$. Let span$_{\mathbb{Q}}\{u(d): d\mid \frac{n}{4}\}$ be the set of all linear combination of $\{u(d): d\mid \frac{n}{4}\}$ over $\mathbb{Q}$. Therefore span$_{\mathbb{Q}}\{u(d): d\mid \frac{n}{4}\}$ is a subspace of $\mathbb{U}$ over $\mathbb{Q}$. Let $ E(\mathbb{U})=\{ Eu: u \in \mathbb{U} \}$. Now we propose the following lemma.

	\begin{lema}\label{iocg11}
		Let $n\equiv 0\Mod 4$. Then $E (\mathbb{U}) \subseteq$ span$_{\mathbb{Q}}\{u(d): d\mid \frac{n}{4}\}$.
	\end{lema}
	\begin{proof}
		Let $u=[u_1,u_2,...,u_{n-1}]^{t} \in U$ and $v=E u=[v_1,v_2,...,v_{n-1}]^t$. Since $G_n(d_1) \cap G_n(d_2) = \emptyset$ for $d_1\neq d_2$ and $\bigcup\limits_{d \mid n}G_n(d)=\{ 1,2,...,n-1\}$, it is enough to show the following:
		\begin{enumerate}[label=(\roman*)]
			\item $v_s=-v_{n-s}$ for all $1\leq s \leq n-1$.
			\item If $d \mid \frac{n}{4}$ then $v_s=v_t$ for all $s,t\in G_n^1(d)$.
			\item If $d \nmid \frac{n}{4}, d \mid n$ and $d<n$ then $v_r=0$ for all $r\in G_n(d)$.
		\end{enumerate} 
		
		Since $Eu=[v_1,v_2,...,v_{n-1}]^t$ for $1\leq s \leq n-1$, we get $$v_{n-s}=\sum\limits_{j=1}^{n-1}iu_jw_n^{(n-s)j}=\sum\limits_{j=1}^{n-1}iu_jw_n^{-sj}= \overline{\sum\limits_{j=1}^{n-1}(-i)u_jw_n^{sj}}=-\overline{v_s}=-v_s.$$
		
		Assume that $d \mid \frac{n}{4}$ and $s,t\in G_n^1(d)$, so that $\frac{s}{d},\frac{t}{d}\in G_{\frac{n}{d}}^1(1)$ and $w_n^s$, $w_n^t$ are roots of $\Phi_{\frac{n}{d}}^1(x)$. Again, $v_{s}=\sum\limits_{j=1}^{n-1}iu_jw_n^{sj} \in \mathbb{Q}$, and so $w_n^s$ is a root of the polynomial $p(x)=\sum\limits_{j=1}^{n-1}iu_jx^j-v_s \in \mathbb{Q}(i)[x]$. Therefore $p(x)$ is a multiple of the irreducible monic polynomial $\Phi_{\frac{n}{d}}^1(x)$, and so $w_n^t$ is also a root of $p(x)$, i.e. $v_{s}=\sum\limits_{j=1}^{n-1}iu_jw_n^{tj}=v_t$. Hence $v_s=v_t$. 
		
		Now assume that $d \nmid \frac{n}{4}, d \mid n$, $d<n$ and $r\in G_n(d)$. Then $r,n-r\in G_n(d)$ and $w_n^r$, $w_n^{n-r}$ are roots of $\Phi_{\frac{n}{d}}(x)$. Again $v_{r}=\sum\limits_{j=1}^{n-1}iu_jw_n^{rj} \in \mathbb{Q}$, and hence $w_n^r$ is a root of polynomial in $q(x)=\sum\limits_{j=1}^{n-1}iu_jx^j-v_r \in \mathbb{Q}(i)[x]$. Therefore $q(x)$ is a multiple of the irreducible monic polynomial $\Phi_{\frac{n}{d}}(x)$. Note that $\Phi_{\frac{n}{d}}(x)$ is irreducible over $\mathbb{Q}(i)$ as $n \not\equiv 0 \Mod 4$.
		Thus $w_n^{n-r}$ is also a root of $q(x)$, $i.e.$, $v_{r}=\sum\limits_{j=1}^{n-1}iu_jw_n^{(n-r)j}=v_{n-r}$.  This gives $v_r=v_{n-r}$. Now $v_r=v_{n-r}$ and $v_r=-v_{n-r}$ altogether gives $v_r=v_{n-r}=0$.
	\end{proof}
	
	The next theorem is crucial in proving the necessity of the sufficient condition of Corollary~\ref{iocg6}. 
	
	\begin{theorem}\label{iocg12}
		Let $n\equiv 0\Mod 4$. Then $\mathbb{U}=$ span$_{\mathbb{Q}}\{u(d): d\mid \frac{n}{4} \}$.
	\end{theorem}
	\begin{proof}
		By Lemma~\ref{iocg11}, $E(\mathbb{U}) \subseteq$ span$_{\mathbb{Q}}\{u(d): d\mid \frac{n}{4} \} \subseteq \mathbb{U}$. Since $E$ is invertible, we get $\dim E(\mathbb{U}) = \dim \mathbb{U}$, and hence $E(\mathbb{U}) =$ span$_{\mathbb{Q}}\{u(d): d\mid \frac{n}{4} \} = \mathbb{U}$.
	\end{proof}

	\begin{theorem}\label{iocg13}
		Let $G=Circ(\mathbb{Z}_n,\mathcal{C})$ be an oriented circulant graph.  
		\begin{enumerate}[label=(\roman*)]
			\item If $n\not\equiv 0 \Mod 4$ then $G$ is integral if and only if $\mathcal{C}=\emptyset$.
			\item If $n\equiv 0 \Mod 4$ then $G$ is integral if and only if $\mathcal{C}= \bigcup\limits_{d\in \mathscr{D}}S_n(d)$, where $\mathscr{D} \subseteq \{ d: d\mid \frac{n}{4}\}$ and $S_n(d) = G_n^{1}(d)$ or $S_n(d)=G_n^{3}(d)$.
		\end{enumerate}
	\end{theorem}
	\begin{proof}
		\begin{enumerate}[label=(\roman*)]
			\item Assume that $G$ is integral. Let $Sp_H(G)=\{\mu_0,\mu_1,...,\mu_{n-1} \}$. By Corollary~\ref{iocg02}, $\mu_j=-\mu_{n-j}\in \mathbb{Q}$ for all $j=0,...,n-1$, where $$\mu_j= i \sum\limits_{k\in {\mathcal{C}}} w_n^{jk} -i \sum\limits_{k\in {\mathcal{C}}} w_n^{-jk}, \textnormal{ for } 0\leq j\leq n-1.$$ Hence $w_n^j$ is a root of polynomial  $p(x)= i \sum\limits_{k\in {\mathcal{C}}} x^{k} -i \sum\limits_{k\in {\mathcal{C}}} x^{-k}-\mu_j \in \mathbb{Q}(i)[x]$. Since $n\not\equiv 0 \Mod 4$, the polynomial $\Phi_n(x)$ is irreducible in $\mathbb{Q}(i)[x]$. Therefore, $p(x)$ is a multiple of the irreducible polynomial $\Phi_n(x)$, and so $w_n^{-j}=w_n^{n-j}$ is also a root of $p(x)$, i.e. $\mu_j=\mu_{n-j}$. Now $\mu_j=\mu_{n-j}$ and $\mu_j=-\mu_{n-j}$ implies that $\mu_j=\mu_{n-j}=0$, for all $j=0,...,n-1$. Hence $\mathcal{C}=\emptyset$. Converse part is easy to prove. 
			
			\item We have already the proved sufficient part in Corollary~\ref{iocg6}. Assume that $G$ is an integral oriented circulant graph with non-empty symbol set $\mathcal{C}$. Let $u$ be the vector of length $n-1$ defined by
			$$u_k= \left\{ \begin{array}{rl}
				1 & \mbox{if }  k\in \mathcal{C} \\
				-1 & \mbox{if }  n-k\in \mathcal{C}\\ 
				0 &  \mbox{otherwise}. 
			\end{array}\right.
			$$ Note that each entry of $Eu$ is an eigenvalue of $G$, and so $Eu \in \mathbb{Z}^{n-1}$. 
			Therefore $u\in \mathbb{U}$, and by Theorem~\ref{iocg12}, $u\in $ span$_{\mathbb{Q}}\{u(d): d\mid \frac{n}{4} \}$. That is,
			$$u=\sum\limits_{ d\mid \frac{n}{4} } c_d u(d) , \textnormal{ where } c_d \in \mathbb{Q} \textnormal{ for } d \mid \frac{n}{4}.$$ 
			
			Note that coordinates of $u$ and $u(d)$ belong to $\{0,\pm 1\}$. Since $G_n(d_1)\cap G_n(d_2)=\emptyset$ for all $d_1\neq d_2$, so $u_k=c_du(d)_k$ for all $k\in G_n(d)$ implies that $c_d\in \{ 0,\pm 1\}$ for all $d\mid \frac{n}{4}$. Let $\mathscr{D}$ be the set of all $d$ for which $c_d\neq 0$ and $d\mid \frac{n}{4}$. Thus $\mathcal{C}= \bigcup\limits_{d\in \mathscr{D}}S_n(d)$, where $S_n(d) = G_n^{1}(d)$ or $S_n(d)=G_n^{3}(d)$.
		\end{enumerate}
	\end{proof}

	\begin{theorem}\label{iocg14}
		Let $G=Circ(\mathbb{Z}_n,\mathcal{C})$ be mixed circulant graph on $n$ vertices with symbol set $\mathcal{C}$. Then $G$ is  integral if and only if  $\mathcal{C} \setminus \overline{\mathcal{C}} = \bigcup\limits_{d\in \mathscr{D}_1}G_n(d)$ and
		\begin{equation*}
			\overline{\mathcal{C}}= \left\{ \begin{array}{ll}
				\emptyset & \textnormal{if } n\not\equiv 0\Mod 4 \\ 
				\bigcup\limits_{d\in \mathscr{D}_2}S_n(d) & \textnormal{if } n\equiv 0\Mod 4,
			\end{array}\right.
		\end{equation*}
		where $\mathscr{D}_1 \subseteq \{ d: d\mid n\}$, $\mathscr{D}_2 \subseteq \{ d: d\mid \frac{n}{4}\}$, $\mathscr{D}_1 \cap \mathscr{D}_2 = \emptyset$, and $S_n(d) = G_n^{1}(d)$ or $S_n(d)=G_n^{3}(d)$.
	\end{theorem}
	\begin{proof} By Theorem~\ref{iocg24}, $G$ is integral if and only if both $(\mathbb{Z}_n,\mathcal{C}\setminus \overline{\mathcal{C}})$ and $(\mathbb{Z}_n, \overline{\mathcal{C}})$ are integral. Therefore the result follows from Theorem~\ref{2006integral} and Theorem~\ref{iocg13}.
	\end{proof}
	
	Wasin So \cite{2006integral} gave an upper bound on the number of integral circulant graphs on $n$ vertices.
	
	\begin{corollary}\label{iocgws}\cite{2006integral}
		Let $\tau(n)$ be the number of divisors of $n$. Then there are at most $2^{\tau(n)-1}$ integral mixed circulant graphs on $n$ vertices.
	\end{corollary}
	
	In the next result, we give a similar upper bound on the number of integral mixed circulant graphs on $n$ vertices.
	
	\begin{corollary}
		Let $\tau(n)$ be the number of divisors of $n$. Then there are at most $k(n)$ integral mixed circulant graphs on $n$ vertices, where
		\begin{equation*}
			k(n)= \left\{ \begin{array}{ll}
				2^{\tau(\frac{n}{4})+\tau(n)-1} & \textnormal{if } n\equiv 0\Mod 4 \\ 
				2^{\tau(n)-1} & \textnormal{if } n \not \equiv 0\Mod 4.
			\end{array}\right.
		\end{equation*}
	\end{corollary}
	\begin{proof}
		If $n \not \equiv 0\Mod 4$ then Theorem \ref{iocg14} says that every integral mixed circulant graph is integral. By Corollary \ref{iocgws}, there are at most $2^{\tau(n)-1}$ integral mixed circulant graphs on $n$ vertices. Assume that $n\equiv 0\Mod 4$ and $Circ(\mathbb{Z}_n,\mathcal{C})$ is a mixed circulant graph on $n$ vertices with symbol set $\mathcal{C}$. For every divisor $d$ of $\frac{n}{4}$, by Theorem \ref{iocg14}, $G_n(d)$, $G_n^1(d)$, $G_n^3(d)$ and $\emptyset$ are subsets of the symbol set $\mathcal{C}$. Thus, we have $4^{\tau{(\frac{n}{4})}}$ distinct symbol sets for an integral mixed circulant graph corresponding to divisors of $\frac{n}{4}$. For each proper divisor $d$ of $n$ which is not a divisor of $\frac{n}{4}$, by Theorem \ref{iocg14}, $G_n(d)$ and $\emptyset$ are  subsets of the symbol set $\mathcal{C}$. Thus, we have $2^{\tau(n)-\tau{(\frac{n}{4})-1}}$ distinct symbol sets for an integral mixed circulant graph corresponding to proper divisors of $n$ which are not divisors of $\frac{n}{4}$.
		Hence we have $4^{\tau{(\frac{n}{4})}}.2^{\tau(n)-\tau{(\frac{n}{4})-1}}=2^{\tau(\frac{n}{4})+\tau(n)-1} $ distinct symbol sets for an integral mixed circulant graph. Since distinct symbol sets may correspond to isomorphic circulant graphs, so  there are at most $2^{\tau(\frac{n}{4})+\tau(n)-1} $  integral mixed circulant graphs on $n$ vertices.
	\end{proof}
	

	\section{ Ramanujan sum and eigenvalues of oriented circulant graph} 
	
	In 1918, Ramanujan \cite{ramanujan1918certain} published a seminal paper in which he introduced a sum (now called Ramanujan sums), for each $n,q \in \mathbb{N}$, defined by
	\begin{equation}\label{ramasum}
		\begin{split}
			c_n(q)= \sum_{a\in G_n(1)} \cos \bigg(\frac{2\pi a q}{n}\bigg).
		\end{split} 
	\end{equation}
	This sum can also be written as  $$c_n(q)= \sum_{a\in G_n(1)} w_n^{aq}=\sum_{\substack{a\in G_n(1) \\ a < \frac{n}{2}}} 2\cos \bigg(\frac{2\pi a q}{n}\bigg),$$ where $w_n= \exp ( \frac{2\pi i}{n})$. A set $\mathcal{C} \subseteq G_n(1)$ is said to be \textit{skew-symmetric} in $G_n(1)$ if $n-a \not\in S $ for all $a\in S$ and $\mathcal{C} \cup \mathcal{C}^{-1}= G_n(1)$. It is clear that a skew-symmetric set in $G_n(1)$ contains exactly $\frac{\varphi(n)}{2}$ elements.
	
	It is easy to see, for any skew-symmetric set $\mathcal{C}$ in $G_n(1)$, that the Ramanujan sum (\ref{ramasum}) can be re-written as 
	\begin{equation}\label{ramasum2}
		\begin{split}
			c_n(q)= \sum_{a\in \mathcal{C}} 2 \cos \bigg(\frac{2\pi a q}{n}\bigg).
		\end{split} 
	\end{equation}
	
	One of the important properties of Ramanujan sum is that $c_n(q)$ is an integer for all $q,n \in \mathbb{N}$. So, it is a natural question to ask, ``if we replace cosine with sine in Ramanujan sum, then is the sum remain an integer or not for all $q,n \in \mathbb{N}$?''.
	Let us replace the cosine by sine in (\ref{ramasum2}), and write
	
	\begin{equation}\label{newsum}
		\begin{split}
			s_n^{\mathcal{C}}(q)= \sum_{a\in \mathcal{C}} 2 \sin \bigg(\frac{2\pi a q}{n}\bigg),
		\end{split} 
	\end{equation}
	where $\mathcal{C}$ is any skew-symmetric set in $G_n(1)$. The sum (\ref{newsum}) can be re-written as 
	\begin{equation}\label{newsum2}
		\begin{split}
			s_n^{\mathcal{C}}(q)= \sum_{a\in \mathcal{C}} \frac{w_n^{aq}-w_n^{-aq}}{i}.
		\end{split} 
	\end{equation}
	
	\begin{theorem}\label{sumchara} The following are true.
		\begin{enumerate}[label=(\roman*)]
			\item For $n \not\equiv 0\Mod 4$, there does not exist skew-symmetric set $\mathcal{C}$ in $G_n(1)$ such that $s_n^{\mathcal{C}}(q) \in \mathbb{Z}$  for all $q,n \in \mathbb{N}$.
			\item For $n \equiv 0\Mod 4$, $s_n^{\mathcal{C}}(q) \in \mathbb{Z}$  for all $q,n \in \mathbb{N}$ if and only if $\mathcal{C}=G_n^1(1)$ or $G_n^3(1)$.
		\end{enumerate}
	\end{theorem} 
	The proof of Theorem \ref{sumchara} follows from Theorem \ref{iocg13}. In what follows, we consider $\mathcal{C}=G_n^1(1)$ for the sum in (\ref{newsum}) and (\ref{newsum2}), and we simply write $s_n^{\mathcal{C}}(q)$ by $s_n(q)$. That is,
	
	\begin{equation*}\label{newsum3}
		\begin{split}
			s_n(q) = \sum_{a\in G_n^1(1)} \frac{w_n^{aq}-w_n^{-aq}}{i}.
		\end{split} 
	\end{equation*}
	
	In Theorem \ref{sumchara}, we found that $s_n(q)$ is integer for all $n,q\in \mathbb{N}$ and $n \equiv (0\mod 4)$.
	
	\begin{theorem}\cite{2006integral}
		Let $G=Circ(\mathbb{Z}_n,\mathcal{C})$ be an integral circulant graph on $n$ vertices with $Sp_H(G)=\{ \lambda_0,..., \lambda_{n-1} \}$. Let $\mathcal{C}=G_n(d)$, where $ d\mid n$. Then 
		$$\lambda_j= c_{n/d}(j), \textnormal{ for } 0\leq j \leq n-1.$$
		Furthermore, $Sp_H(G)$ is equal to $d$ copies of $(c_{n/d}(0),c_{n/d}(1),...,c_{n/d}((n/d)-1))$. 
	\end{theorem}

	\begin{theorem}\label{OriEigInt}
		Let $G=Circ(\mathbb{Z}_n,\mathcal{C})$ be an integral oriented circulant graph on $n\equiv 0 \Mod 4$ vertices with $Sp_H(G)=\{ \mu_0,..., \mu_{n-1} \}$. Let $\mathcal{C}=G_n^1(d)$, where $ d\mid \frac{n}{4}$. Then 
		$$\mu_j= -s_{n/d}(j), \textnormal{ for } 0\leq j \leq n-1.$$
		Furthermore, $Sp_H(G)$ is equal to $d$ copies of $(-s_{n/d}(0),-s_{n/d}(1),...,-s_{n/d}((n/d)-1))$. 
	\end{theorem}
	\begin{proof}
		Using Corollary~\ref{iocg02}, we get
		\begin{equation*}\label{}
			\begin{split}
				\mu_j = i \sum\limits_{k\in {\mathcal{C}}} (w_n^{jk} - w_n^{-jk}) = - \sum_{a\in G_n^1(d)} \frac{w_n^{aj}-w_n^{-aj}}{i} &= - \sum_{a\in dG_n^1(1)} \frac{w_n^{aj}-w_n^{-aj}}{i}\\
				&= - \sum_{a\in G_n^1(1)} \frac{(w_n^d)^{aj}-(w_n^d)^{-aj}}{i}\\
				&= -s_{n/d}(j).
			\end{split} 
		\end{equation*}
		Since $s_{n/d}(t_1)=s_{n/d}(t_2)$ for $t_1 \equiv t_2 (\mathrm{mod}\hspace{0.1cm} \frac{n}{d})$, the spectrum $\{ \mu_1,...,\mu_{n-1} \}$ is equal to $d$ copies of \linebreak[4] $(-s_{n/d}(0),-s_{n/d}(1),...,-s_{n/d}((n/d)-1))$.  
	\end{proof}
	
	\begin{corollary} 
		Let $G=Circ(\mathbb{Z}_n,\mathcal{C})$ be an integral oriented circulant graph on $n\equiv 0 \Mod 4$ vertices  with $Sp_H(G)=\{ \mu_0,..., \mu_{n-1} \}$. Then $$\mu_j= \sum_{d\in \mathscr{D}} \pm s_{n/d}(j), \textnormal{ for } 0\leq j \leq n-1, \textnormal{ where } \mathscr{D} \subseteq \left\{ d: d\mid \frac{n}{4}\right\}.$$ 
	\end{corollary}
	\begin{proof} By Theorem \ref{iocg13}, $\mathcal{C}= \bigcup\limits_{d\in \mathscr{D}}S_n(d)$, where $\mathscr{D} \subseteq \{ d: d\mid \frac{n}{4}\}$ and $S_n(d) = G_n^{1}(d)$ or $S_n(d)=G_n^{3}(d)$.
		Using Corollary~\ref{iocg02}, we get
		\begin{equation*}\label{}
			\begin{split}
				\mu_j &= i \sum\limits_{k\in {\mathcal{C}}} (w_n^{jk} - w_n^{-jk}) = \sum_{d \in \mathscr{D}}  \sum_{a\in S_n(d)} - \frac{w_n^{aj}-w_n^{-aj}}{i} = \sum_{d\in \mathscr{D}} \pm s_{n/d}(j).\\
			\end{split} 
		\end{equation*}
	\end{proof}
	
	\begin{corollary}
		Let $G=Circ(\mathbb{Z}_n,\mathcal{C})$ be integral mixed circulant graph on $n$ vertices with $Sp_H(G)=\{ \gamma_0,..., \gamma_{n-1} \}$. Then 
		$$\gamma_j=\sum_{d\in \mathscr{D}_1} c_{n/d}(j) + \sum_{d\in \mathscr{D}_2} \pm s_{n/d}(j), \textnormal{ for } 0\leq j \leq n-1,$$ where $\mathscr{D}_1 \subseteq \{ d: d\mid n\}$, $\mathscr{D}_2 \subseteq \{ d: d\mid \frac{n}{4}\}$, and $\mathscr{D}_1 \cap \mathscr{D}_2 = \emptyset$.
	\end{corollary}
	\begin{proof} By Theorem \ref{iocg14}, $\mathcal{C} \setminus \overline{\mathcal{C}} = \bigcup\limits_{d\in \mathscr{D}_1}G_n(d)$ and
		$$\overline{\mathcal{C}}= \left\{ \begin{array}{lll}
			\emptyset & \mbox{if}
			& n\not\equiv 0\Mod 4 \\ \bigcup\limits_{d\in \mathscr{D}_2}S_n(d) & \mbox{if} & n\equiv 0\Mod 4, 
		\end{array}\right. $$ where $\mathscr{D}_1 \subseteq \{ d: d\mid n\}$, $\mathscr{D}_2 \subseteq \{ d: d\mid \frac{n}{4}\}$, $\mathscr{D}_1 \cap \mathscr{D}_2 = \emptyset$, and $S_n(d) = G_n^{1}(d)$ or $S_n(d)=G_n^{3}(d)$. Using Lemma \ref{iocg03}, we get
		\begin{equation*}\label{}
			\begin{split}
				\gamma_j & =  \sum_{d \in \mathscr{D}_1} \sum_{a\in G_n(d)} w_n^{aj} +  \sum_{d \in \mathscr{D}_2}  \sum_{a\in S_n(d)} - \frac{w_n^{aj}-w_n^{-aj}}{i} = \sum_{d\in \mathscr{D}_1}  c_{n/d}(j) + \sum_{d\in \mathscr{D}_2} \pm s_{n/d}(j).
			\end{split} 
		\end{equation*}
	\end{proof}
	
	From Theorem~\ref{OriEigInt}, it is clear that $s_n(q)$ play an important role in the eigenvalues of integral mixed circulant graph. In the next result we see some basic properties of $s_n(q)$.
	
	\begin{theorem} The following are some basic properties of $s_n(q)$.
		\begin{enumerate}[label=(\roman*)]
			\item $s_n(t_1)=s_n(t_2)$ for $t_1 \equiv t_2 \Mod n$.
			\item $s_n(n-t)=-s_n(t)$.
			\item $s_n(0)= s_n(\frac{n}{2})= 0$.
			\item $s_n(\frac{n}{4})= \varphi(n)$.
			\item $s_n(\frac{3n}{4})=- \varphi(n)$.
		\end{enumerate}
	\end{theorem}

	\begin{prop}
		Let $n=2^k$ and $k\geq 2$. Then 
		$$s_n(t)= \left\{ \begin{array}{cl}
			2^{k-1}& \textnormal{ if $t\equiv  2^{k-2} \Mod n$ } \\ 
			-2^{k-1}& \textnormal{ if $t\equiv 3.2^{k-2} \Mod n$ } \\
			0 & \textnormal{ otherwise.} 
		\end{array}\right. $$
	\end{prop}
	\begin{proof}
		Since $n=2^k$ implies that $G_n^1(1)= M_n^1(1)$, so 
		\begin{equation*}
			\begin{split}
				s_n(t) &=\sum_{a\in M_n^1(1) } \frac{w_n^{at}-w_n^{-at}}{i}\\
				&= \frac{1}{i} (w_n^t -w_n^{3t}+ w_n^{5t}-...-w_n^{(n-1)t} ).
			\end{split} 
		\end{equation*}
		As $s_n(t)$ is a sum of a geometric progression with common ratio $-w_n^{2t}$, now it follows the required result.
	\end{proof}

	The classical M$\ddot{\text{o}}$bius function $\mu(n)$ is defined by 
	$$\mu(n)= \left\{ \begin{array}{cl}
		1& \mbox{if } n=1 \\ 
		(-1)^k & \mbox{if }  n=p_1...p_k \textnormal{ with $p_i$ distinct prime} \\
		0   &\mbox{otherwise.} 
	\end{array}\right. $$
	
	Let $\delta$ be the indicator function defined by
	$$\delta(n)= \left\{ \begin{array}{lll}
		1& \mbox{if}
		& n=1 \\ 0 & \mbox{if} & n>1 . 
	\end{array}\right. $$
	
	\begin{theorem} \cite{murty2008problems} $$ \sum_{d\mid n} \mu(d) =\delta(n).$$
	\end{theorem}
	
	E. Cohen \cite{cohen} introduced a generalized M$\ddot{\text{o}}$bius inversion formula of arbitrary direct factor sets. Let $P,Q\neq \emptyset$ be subsets of N such that if $n_1,n_2 \in \mathbb{N}$, $(n_1,n_2) = 1$, then $n = n_1n_2 \in P$ (resp. $n = n_1n_2 \in Q$) if and only if $ n_1,n_2 \in P$ (resp. $ n_1,n_2 \in Q$). If each integer $n \in \mathbb{N}$ possesses a unique factorization of the form $n = ab$  $(a \in P, b \in Q)$, then each of the sets $P$ and $Q$ is called a direct factor set of $\mathbb{N}$. In what follows $P$ will denote such a direct factor set with (conjugate) factor set $Q$.
	The M$\ddot{\text{o}}$bius function can be generalized to arbitrary direct factor sets, by writing $$\mu_P(n)= \sum_{d\mid n,d\in P} \mu\bigg(\frac{n}{d}\bigg),$$ where $\mu$ is the classical M$\ddot{\text{o}}$bius function. For example, $\mu_1(n) = \mu_{\{1\}}(n)$ and $\mu_{\mathbb{N}}(n) =\delta(n)$. 
	
	\begin{theorem}\cite{cohen} $$ \sum_{d\mid n, d\in Q} \mu_P\bigg( \frac{n}{d} \bigg)=\delta(n).$$
	\end{theorem}
	
	\begin{theorem}\cite{cohen}\label{iocg31} If $f(n)$ and $g(n)$ are arithmetic functions then 
		$$f(n)= \sum_{d\mid n, d\in Q} g\bigg(\frac{n}{d} \bigg) \textnormal{ if and only if } g(n)= \sum_{d\mid n} f(d) \mu_P\bigg( \frac{n}{d} \bigg).$$
	\end{theorem}
	
	Let us now take the direct factor $P=\{2^k: k\geq 0 \}$, and $Q$ the set of all odd natural numbers.

	\begin{theorem}\label{sumformula} 
		Let $n\equiv 0 \mod 4$ and $D_n^3=\emptyset$. Then $$s_n(t)=2 \delta \mu_P(n_1) \sum_{\substack{e \mid \delta \\ \frac{te}{\delta} \textnormal{ is odd}\\ \gcd(n_1,e)=1 }} (-1)^{\frac{te-\delta}{2\delta}} \frac{\mu_P(e)}{e},$$ where $\delta = \gcd(\frac{n}{4},t)$, $n_1=\frac{n}{4\delta}$.
	\end{theorem}
	\begin{proof}
		Let \begin{equation*}
			\begin{split}
				f_n(t) &=\sum_{a\in M_n^1(1) } \frac{w_n^{at}-w_n^{-at}}{i}\\
				&= \frac{1}{i} (w_n^t -w_n^{3t}+ w_n^{5t}-...-w_n^{(n-1)t} )
			\end{split} 
		\end{equation*}
		As $f_n(t)$ is a sum of a geometric progression with common ratio $-w_n^{2t}$, we find that 
		$$f_n(t)= \left\{ \begin{array}{rl}
			\frac{n}{2}& \mbox{ if }  t\equiv \frac{n}{4} \Mod n \\
			-\frac{n}{2}& \mbox{ if }  t\equiv \frac{3n}{4} \Mod n \\
			0 &\mbox{ otherwise.} 
		\end{array}\right. $$
		Using Lemma~\ref{iocg4}, we have
		\begin{equation*}
			\begin{split}
				f_n(t)= \sum_{a\in M_n^1(1) } \frac{w_n^{at}-w_n^{-at}}{i}=  \sum_{d\in D_n^1} \sum_{a\in G_n^1(d) } \frac{w_n^{at}-w_n^{-at}}{i} &=  \sum_{d\in D_n^1} \sum_{a\in dG_{n/d}^1(1) } \frac{w_n^{at}-w_n^{-at}}{i} \\
				&= \sum_{d\in D_n^1} \sum_{a\in G_{n/d}^1(1) } \frac{(w_n^d)^{at}-(w_n^d)^{-at}}{i} \\
				&= \sum_{d\in D_n^1} s_{n/d}(t).
			\end{split} 
		\end{equation*} In the last equation, we are using that $w^d=\exp(\frac{2\pi i}{n/d})$ is $\frac{n}{d}$-th root of unity. From Theorem~\ref{iocg31}, 
		\begin{equation*}
			\begin{split}
				s_n(t)&=\sum_{\substack{d \mid n \\ t\equiv \frac{d}{4} \Mod d }} f_d(t) \mu_P\bigg(\frac{n}{d}\bigg) + \sum_{\substack{d \mid n \\ t\equiv \frac{3d}{4} \Mod d }} f_d(t) \mu_P\bigg(\frac{n}{d}\bigg)\\
				&=\sum_{\substack{d \mid n \\ t\equiv \frac{d}{4} \Mod d }} \frac{d}{2} \mu_P\bigg(\frac{n}{d}\bigg) - \sum_{\substack{d \mid n \\ t\equiv \frac{3d}{4} \Mod d }} \frac{d}{2} \mu_P\bigg(\frac{n}{d}\bigg)\\
				&=\sum_{\substack{4d \mid n \\ t\equiv d \Mod {4d} }} 2d \mu_P\bigg(\frac{n}{4d}\bigg) - \sum_{\substack{4d \mid n \\ t\equiv 3d \Mod {4d} }} 2d \mu_P\bigg(\frac{n}{4d}\bigg)\\
				&=\sum_{\substack{d \mid \frac{n}{4} \\ \frac{t}{d}\equiv 1 \Mod {4} }} 2d \mu_P\bigg(\frac{n}{4d}\bigg) - \sum_{\substack{d \mid \frac{n}{4} \\ \frac{t}{d} \equiv 3 \Mod {4} }} 2d \mu_P\bigg(\frac{n}{4d}\bigg)\\
				&=\sum_{\substack{d \mid \gcd(\frac{n}{4},t) \\ \frac{t}{d} \textnormal{ is an odd} }} (-1)^{\frac{t-d}{2d}} 2d \mu_P\bigg(\frac{n}{4d}\bigg).\\
			\end{split} 
		\end{equation*} Let $\delta = \gcd(\frac{n}{4},t)$ and $\frac{n}{4}=\delta n_1$. We have 
		
		\begin{equation*}
			\begin{split}
				s_n(t) =\sum_{\substack{d \mid \delta \\ \frac{t}{d} \textnormal{ is odd} }} (-1)^{\frac{t-d}{2d}} 2d \mu_P\bigg(\frac{n}{4d}\bigg) &=\sum_{\substack{de = \delta \\ \frac{t}{d} \textnormal{ is odd} }} (-1)^{\frac{t-d}{2d}} 2d \mu_P(n_1e)\\
				&=2 \mu_P(n_1) \sum_{\substack{de = \delta \\ \frac{t}{d} \textnormal{ is odd}\\ \gcd(n_1,e)=1 }} (-1)^{\frac{t-d}{2d}} d \mu_P(e)\\
				&=2 \delta \mu_P(n_1) \sum_{\substack{e \mid \delta \\ \frac{te}{\delta} \textnormal{ is odd}\\ \gcd(n_1,e)=1 }} (-1)^{\frac{te-\delta}{2\delta}} \frac{\mu_P(e)}{e}.
			\end{split} 
		\end{equation*}
		
	\end{proof}
	
	\begin{corollary} Let $n \equiv 0 \Mod 4$, $D_n^3=\emptyset$. Then
		$$s_n(t)= \left\{ \begin{array}{cl}
			(-1)^{\frac{t-1}{2}}2 \mu_P( \frac{n}{4})& \textnormal{ if $t$ is odd}\\ 
			0& \textnormal{ if $t$ is even,}
		\end{array}\right. $$ where $\gcd(\frac{n}{4},t)=1$.
	\end{corollary}
	\begin{proof}
		Apply $\delta=\gcd(\frac{n}{4},t)=1$ in Theorem \ref{sumformula}.
	\end{proof}
	

\end{document}